\documentclass[11pt,a4paper]{amsart}

\usepackage{amsmath,amssymb,amsthm}
\usepackage{euscript}
\usepackage{a4wide}
\usepackage{enumitem}
\usepackage{url}
\usepackage[colorlinks=true,linkcolor=blue,citecolor=blue,urlcolor=blue]{hyperref}
\usepackage{comment}

\theoremstyle{plain}
\newtheorem{theorem}{Theorem}[section]
\newtheorem{proposition}[theorem]{Proposition}
\newtheorem{lemma}[theorem]{Lemma}
\newtheorem{corollary}[theorem]{Corollary}

\theoremstyle{definition}
\newtheorem{definition}[theorem]{Definition}
\newtheorem{remark}[theorem]{Remark}

\newtheorem{question}[theorem]{Question}

\newtheorem{theoremA}{Theorem}

\numberwithin{equation}{section}

\newcommand{\Bop}{\EuScript{B}}
\newcommand{\ME}{\EuScript{M}_E}
\newcommand{\MJ}{\EuScript{M}_J}
\newcommand{\ER}{E_{\mathrm R}}
\newcommand{\IR}{\EuScript{I}_{\mathrm R}}
\newcommand{\JR}{\EuScript{J}_{\mathrm R}}
\newcommand{\Kideal}{\EuScript{K}}
\newcommand{\Wideal}{\EuScript{W}}
\newcommand{\ideal}[1]{\EuScript{#1}}
\newcommand{\Id}{\mathrm{Id}}
\newcommand{\inv}{\mathrm{inv}}
\newcommand{\C}{\mathbb C}
\newcommand{\K}{\mathbb K}
\newcommand{\N}{\mathbb N}
\newcommand{\norm}[1]{\lVert #1\rVert}

\newcommand{\cpi}{C_{\mathrm{pi}}}
\DeclareMathOperator{\ind}{ind}

\renewcommand{\le}{\leqslant}
\renewcommand{\ge}{\geqslant}
\renewcommand{\leq}{\leqslant}
\renewcommand{\geq}{\geqslant}

\dedicatory{Dedicated to Niels Jakob Laustsen, whose mentorship shaped all three authors.}
\begin{document}

\title[Pure infiniteness and primary factorisation]{Pure infiniteness and primary factorisation}

\author[A. Acuaviva]{Antonio Acuaviva}
\address[A. Acuaviva]{School of Mathematical Sciences,
Fylde College,
Lancaster University,
LA1 4YF,
United Kingdom}
\email{ahacua@gmail.com}

\author[B. Horv\'ath]{Bence Horv\'ath}
\address[B. Horv\'ath]{Helvetia Insurance Ltd., St. Alban-Anlage~26, 4002 Basel, Switzerland}
\email{hotvath@gmail.com}

\author[T.~Kania]{Tomasz Kania}
\address[T.~Kania]{Mathematical Institute\\Czech Academy of Sciences\\\v Zitn\'a 25 \\115 67 Praha 1\\Czech Republic  and  Institute of Mathematics and Computer Science\\ Jagiellonian University\\ {\L}ojasiewicza 6, 30-348 Krak\'{o}w, Poland
}
\email{kania@math.cas.cz, tomasz.marcin.kania@gmail.com}
\thanks{RVO: 67985840.}

\date{\today}

\begin{abstract}
We show that there is no real or complex indecomposable Banach space with the primary factorisation property (PFP). We relate the PFP of a Banach space $E$ to ring-theoretic infiniteness of $\Bop(E)$ and of $\Bop(E)/\ME$, where $\ME$ denotes the set of operators not factoring the identity on $E$, in the case it is the unique maximal ideal of $\Bop(E)$. For complex $E$ with the PFP, this quotient is purely infinite exactly when it is not scalar. We isolate the quantitative gap relevant to ultrapowers, identify classical sequence spaces as positive non-scalar cases, and show that Read's space $\ER$ does not have the uniform PFP.
\end{abstract}

\maketitle

\section{Introduction}

The primary factorisation property is a strong operator-theoretic form of primariness.  A~Banach space $E$ has this property when, for every operator $T\in\Bop(E)$, the identity on $E$ factors through $T$ or through $\Id_E-T$.  Thus every operator determines a factorisation dichotomy.  Such dichotomies are often used to prove that a Banach space is primary, but they also have consequences for the ring structure of $\Bop(E)$ and of its Calkin-type quotients.

The starting point of this note is the ideal-like set
\[
  \ME=\{T\in\Bop(E):\ \Id_E\hbox{ does not factor through }T\}.
\]
Under the primary factorisation property this set is the largest proper two-sided ideal of $\Bop(E)$.  The quotient
\[
  A_E:=\Bop(E)/\ME
\]
is therefore a unital Banach algebra. This object will be of central importance in the proof of our first result, which establishes that the PFP and indecomposability are mutually exclusive properties in both real and complex Banach spaces:

\begin{theoremA}\label{thm:intro-no_indec_pfp}
    There is no real or complex indecomposable Banach space with the PFP.
\end{theoremA}

As the first and third-named authors already noted at the beginning of~\cite[Section~2.2]{AcuavivaKania2026}, the PFP is intimately related to the primeness and primariness properties of Banach spaces. Indeed, \textit{en route} to proving primariness of a Banach space, the first step is often to show that it has the PFP. It is thus worth contrasting the above theorem with a deep result of Gowers and Maurey from~\cite{GowersMaurey}, who constructed an indecomposable, prime Banach space.
\smallskip

It is natural to ask which ring-theoretic infiniteness properties of $\Bop(E)$ and $A_E$ are forced by the factorisation dichotomy.

Our next result collects the elementary infiniteness consequences of the PFP.  We use the definition of pure infiniteness for unital complex Banach algebras adopted by Daws and the second-named author: $A$ is purely infinite when $A\not\simeq\C$ and, for each non-zero $a\in A$, there are $b,c\in A$ such that $bac=1$.

\begin{theoremA}\label{thm:intro-main}
Let $E$ be a Banach space.
\begin{enumerate}[label=\textup{(\roman*)}]
\item If $E$ is infinite-dimensional and has the primary factorisation property, then $\Bop(E)$ is Dedekind-infinite.
\item The preceding conclusion cannot in general be strengthened to proper infiniteness: $C[0,\omega_1]$ has the primary factorisation property, but $\Bop(C[0,\omega_1])$ is not properly infinite.
\item If $E$ is complex and has the primary factorisation property, then every non-zero element of $A_E=\Bop(E)/\ME$ has the identity factoring through it.  Consequently, $A_E$ is purely infinite if and only if $A_E\not\simeq\C$.
\end{enumerate}
\end{theoremA}

The proof of part~\textup{(i)} is very short: applying the factorisation dichotomy to a rank-one projection forces the identity to factor through the complementary projection, and hence produces an idempotent which is Murray--von Neumann equivalent to the identity but is not equal to it.  For part~\textup{(ii)}, the obstruction is the Loy--Willis quotient.  The third-named author and Laustsen identified the corresponding codimension-one ideal as the set of operators through which the identity does not factor, so $\Bop(C[0,\omega_1])/\ME\simeq\C$; proper infiniteness would pass to this quotient, but $\C$ is not properly infinite.  Part~\textup{(iii)} is proved by passing any factorisation of $\Id_E$ through an operator representative to the quotient.

Read's space gives an instructive negative example for the uniform version.

\begin{theoremA}\label{thm:intro-Read-not-UPFP}
Read's space $\ER$ does not have the uniform primary factorisation property.
\end{theoremA}

The proof is not a gliding-hump construction.  It uses the codimension-one ideal isolated by Read.  Modulo the weakly compact operators this ideal becomes a non-zero square-zero ideal, and such a square-zero layer is incompatible with the uniform estimates in the UPFP.

The remaining issue concerns ultrapowers.  For a unital Banach algebra $A$, put
\[
  \cpi^A(a)=\inf\{\norm b\norm c:\ b,c\in A,\ bac=1\}\qquad(a\neq0),
\]
with value $\infty$ if no such $b,c$ exist. Daws and the second-named author showed that, for a~countably incomplete ultrafilter $\mathcal U$, the ultrapower $(A)^{\mathcal U}$ is purely infinite precisely when $\cpi^A$ is uniformly bounded on the unit sphere, provided $A\not\cong\C$.

If $E$ has the uniform primary factorisation property with constant $C$, then the quotient $A_E$ satisfies only the weaker assertion
\[
  \cpi^{A_E}(a)\le C\quad\hbox{or}\quad \cpi^{A_E}(1-a)\le C.
\]
This does not formally yield a uniform bound for $\cpi^{A_E}(a)$ itself.  We therefore formulate the non-scalar ultrapower problem separately.  The scalar-quotient examples $C[0,\omega_1]$ and $J$ should be viewed as boundary cases showing that the non-scalar hypothesis is essential, not as satisfying resolutions of the ultrapower problem.

The classical sequence spaces lie on the other side.  For $E=c_0(\Gamma)$ or $E=\ell_p(\Gamma)$, $1\le p<\infty$, with $\kappa=|\Gamma|$, one has $\ME=\Kideal_\kappa(E)$, and the quotient $\Bop(E)/\Kideal_\kappa(E)$ has pure infiniteness constant one on its unit sphere.  Consequently, every countably incomplete ultrapower of this quotient is purely infinite.  These examples emphasise that a negative solution of the non-scalar question would require genuinely non-uniform complemented-copy behaviour.

The final new test case is Read's space $\ER$.  Read constructed $\ER$ so that $\Bop(\ER)$ admits a~discontinuous derivation, and hence a discontinuous homomorphism into a Banach algebra.  The key point for the present paper is that Read's codimension-one ideal has a large square-zero image modulo the weakly compact operators.  We prove the following general obstruction: if $\Bop(E)$ has a codimension-one ideal $\ideal I$ and a closed ideal $\ideal N\subsetneq\ideal I$ with $\ideal I^2\subseteq\ideal N$, then $E$ cannot have the UPFP.  Applying this to Read's ideals shows that $\ER$ does not have the UPFP.

Laustsen and Skillicorn later described the weak Calkin algebra by a strongly split exact sequence
\[
  0\longrightarrow \Wideal(\ER)\longrightarrow \Bop(\ER)\longrightarrow \widetilde{\ell}_2\longrightarrow0,
\]
where $\widetilde{\ell}_2$ is the unitisation of the Hilbert space $\ell_2$ with zero product.  Using this description, we also show that the remaining PFP question for $\ER$ is equivalent to the assertion that every weakly compact perturbation $\Id_{\ER}-W$ fixes a complemented copy of $\ER$ in the appropriate factorisation sense.  If this assertion were proved, Read's space would answer a question of first- and third-named authors by giving a Banach space with the PFP but without the UPFP.  It would not, however, solve the non-scalar ultrapower question because the resulting quotient would be $\C$.

\section{Preliminaries}

Throughout, $E$ denotes a Banach space over $\K\in\{\mathbb R,\C\}$. The Banach space of bounded linear operators from a Banach space~$E$ to a Banach space~$F$ is denoted by~$\Bop(E,F)$. We put~$\Bop(E):=\Bop(E,E)$, which is a unital Banach algebra with multiplicative identity~$\Id_E$, where~$\Id_E$ is the identity operator on $E$.  We say that an operator $S\in\Bop(E)$ \emph{factors through} $T\in\Bop(E)$ if there are $U,V\in\Bop(E)$ such that
\[
  S=UTV.
\]

\begin{definition}
A Banach space $E$ has the \emph{primary factorisation property}, abbreviated PFP, if for every $T\in\Bop(E)$ the identity $\Id_E$ factors through $T$ or through $\Id_E-T$.

For $C\ge1$, the space $E$ has the \emph{$C$-primary factorisation property}, abbreviated $C$-PFP, if for every $T\in\Bop(E)$ there are $U,V\in\Bop(E)$ with
\[
  UTV=\Id_E\qquad\hbox{or}\qquad U(\Id_E-T)V=\Id_E
\]
and $\norm U\norm V\le C$.  The space $E$ has the \emph{uniform primary factorisation property}, abbreviated UPFP, if it has the $C$-PFP for some $C\ge1$.
\end{definition}

We record an elementary lemma here that will be used repeatedly throughout the paper.

\begin{lemma}\label{lemma: idempotents_from_factorisation}
    Let~$E$ be a Banach space and let~$T,U,V\in\Bop(E)$ such that~$\Id_E=UTV$. Then~$P:=VUT$ and~$Q:=TVU$ are idempotents in~$\Bop(E)$ with~$P[E] \simeq E \simeq Q[E]$, and~$\ker T \subseteq \ker P$ and~$Q[E] \subseteq T[E]$.
\end{lemma}

\begin{proof}
    The containments~$\ker T \subseteq \ker P$ and~$Q[E] \subseteq T[E]$ are immediate from the definitions of~$P$ and~$Q$. We note that
    \[
    P^2 = (VUT)(VUT) = V(UTV)UT = VUT = P,
    \]
    proving the idempotent property for~$P$. Set
    \[
    A:=UTP|_{P[E]}\quad \text{and} \quad B:=P|^{P[E]}V,
    \]
    it is clear that~$A \in \Bop(P[E],E)$ and~$B \in \Bop(E, P[E])$. We also see that
    \[
    AB = UTP|_{P[E]} P|^{P[E]}V = UTPV = UT(VUT)V = (UTV)(UTV) = \Id_E,
    \]
    and
    \[
    BA = P|^{P[E]}VUTP|_{P[E]} = P|^{P[E]}PP|_{P[E]} = \Id_{P[E]}.
    \]
    This yields~$P[E] \simeq E$, as claimed. The isomorphism~$Q[E] \simeq E$ follows from an analogous argument.
\end{proof}

For a Banach space $E$, set
\begin{equation}\label{eq:ME-definition}
  \ME=\{T\in\Bop(E):\ \Id_E\hbox{ does not factor through }T\}.
\end{equation}
The following elementary observation is useful throughout.

\begin{proposition}\label{prop:PFP-maximal-ideal}
Let~$E$ be a Banach space. The set~$\ME$ is closed in the operator norm. The Banach space~$E$ has the PFP if and only if~$\ME$ is the largest proper ideal of~$\Bop(E)$. In this case $\ME$ is the unique maximal ideal of~$\Bop(E)$.
\end{proposition}

\begin{proof}
We first show that set~$\ME$ is closed in the operator norm. Indeed, suppose~$T_n\to T$ in norm and~$T\notin\ME$. Choose~$A,B\in\Bop(E)$ such that~$ATB=\Id_E$. Then~$AT_nB\to\Id_E$, and as the group of invertible elements in~$\Bop(E)$ is an open set in the operator norm, $AT_nB$ is invertible for all sufficiently large~$n \in \mathbb{N}$. Consequently
\[
\Id_E = (AT_nB)^{-1}(AT_nB) = ((AT_nB)^{-1}A)T_nB,
\]
proving that $\Id_E$ factors through $T_n$ for all sufficiently large $n\in \mathbb{N}$.

Now suppose~$E$ has the PFP.  Plainly $\Id_E\notin\ME$.  The set $\ME$ is stable under left and right multiplication by arbitrary operators: if $\Id_E$ factors through $ATB$, then it factors through $T$.

We check additivity.  Let $S,T\in\ME$ and suppose that $S+T\notin\ME$.  Then there are $A,B\in\Bop(E)$ such that
\[
  A(S+T)B=\Id_E.
\]
Apply the PFP to $ASB$.  Since
\[
  \Id_E-ASB=ATB,
\]
the identity factors through $ASB$ or through $ATB$, and therefore through $S$ or through $T$.  This contradicts $S,T\in\ME$.  Hence $S+T\in\ME$, and $\ME$ is a proper two-sided ideal.

Finally, if $\ideal{J}$ is a proper ideal of $\Bop(E)$ and $T\in\ideal{J}$, then $T\in\ME$; otherwise $\Id_E$ would factor through $T$, forcing $\Id_E\in\ideal{J}$.  Thus every proper ideal is contained in $\ME$, so $\ME$ is the largest proper ideal.

Conversely, suppose that $\ME$ is a maximal ideal.  If the PFP failed, then for some $T\in\Bop(E)$ both $T$ and $\Id_E-T$ would belong to $\ME$.  Since $\ME$ is an ideal, it is a linear subspace, and hence
\[
  \Id_E=T+(\Id_E-T)\in\ME,
\]
contradicting properness.
\end{proof}

We recall the ring-theoretic terminology used below. In a unital algebra~$A$, the group of invertible elements is denoted by~$\inv A$. If~$A$ is a unital Banach algebre, then~$\inv A$ is an open subset of~$A$.

For two idempotents $p$ and~$q$ in a unital algebra $A$, we write $p\sim q$ when they are \textit{algebraically Murray--von Neumann equivalent}, that is, when there are~$x,y\in A$ such that~$p=xy$ and~$q=yx$. The idempotents~$p,q \in A$ are \textit{orthogonal} if~$pq=0=qp$.

\begin{definition}
A unital algebra $A$ is \emph{Dedekind-infinite} if there are $a,b\in A$ such that
\[
  ab=1_A\qquad\hbox{and}\qquad ba\neq1_A.
\]
Equivalently, $1_A$ is Murray--von Neumann equivalent to an idempotent different from $1_A$.

A unital algebra $A$ is \emph{properly infinite} if there are orthogonal idempotents $p,q\in A$ such that
\[
  p\sim1_A\sim q.
\]
Thus proper infiniteness implies Dedekind-infiniteness.
\end{definition}

Proper infiniteness is preserved by unital homomorphisms.  Indeed, the images of the two orthogonal idempotents witnessing proper infiniteness remain orthogonal idempotents equivalent to the identity.

For complex Banach algebras we use the following definition of pure infiniteness.

\begin{definition}\label{def:pure-infinite}
A unital complex Banach algebra $A$ is \emph{purely infinite} if $A\not\simeq\C$ and, for every non-zero $a\in A$, there exist $b,c\in A$ such that
\[
  bac=1_A.
\]
\end{definition}

This is the definition used by Daws and the second-named author for unital Banach algebras \cite{DawsHorvath2022}.  It implies simplicity: if a non-zero ideal contains $a$, then it contains $1_A$.

\section{PFP and indecomposability through a Fredholm-theoretic lens}

The main result in this section is that an indecomposable Banach space (real or complex) cannot have the PFP. We first record a generic lemma that will be useful later on. In what follows $\mathbb{H}$ denotes the quaternions.

\begin{lemma}\label{lemma: same_component}
    Let $\mathbb{K} \in \{\mathbb{R}, \mathbb{C} \}$ and let $\mathbb{F} \in \{\mathbb{R}, \mathbb{C}, \mathbb{H} \}$. Let~$E$ be a Banach space over~$\mathbb{K}$ such that there exists a continuous, unital algebra homomorphism~$\varphi \colon \Bop(E) \to \mathbb{F}$. Let~$T \in \Bop(E) \setminus \ker \varphi$. There exists a continuous map~$\Gamma_T \colon [0,1] \to \Bop(E) \setminus \ker \varphi$ with $\Gamma_T(1)= \varepsilon \cdot \Id_E$ and~$\Gamma_T(0)= T$, where $\varepsilon \in \{1, -1\}$.
\end{lemma}

\begin{proof}
    Set $\varepsilon:= -1$ if $\varphi(T) \in (-\infty, 0)$ and set $\varepsilon := 1$ if $\varphi(T) \in \mathbb{F} \setminus (-\infty, 0)$. We define the map
    \[
    \Gamma_T \colon [0,1] \to \Bop(E); \quad t \mapsto t \cdot \varepsilon \cdot \Id_E + (1-t) \cdot T,
    \]
    which is clearly continuous. Note that~$\Gamma_T(0)=T$ and~$\Gamma_T(1)= \varepsilon \cdot \Id_E$. Also, from the linearity of~$\varphi$ and the identity~$\varphi(\Id_E)=1_{\mathbb{F}}$ we obtain
    \[
    \varphi(\Gamma_T(t)) = t \cdot \varepsilon \cdot 1_{\mathbb{F}} + (1-t) \cdot \varphi(T) \qquad (t \in (0,1)).
    \]
    Hence $\varphi(\Gamma_T(t)) = 0$ if and only if $\varphi(T) = \varepsilon \cdot t/(t-1) 1_{\mathbb{F}}$, which is not possible for any $t \in (0,1)$ due to the choice of $\varepsilon$ depending on $\varphi(T)$.
\end{proof}

Let~$E$ be a Banach space. We recall that~$T \in \Bop(E)$ is a \textit{Fredholm operator} if~$\ker T$ and~$E/T[E]$ are finite-dimensional. Throughout,~$\Phi(E)$ denotes the set of Fredholm operators on the Banach space~$E$. If~$T \in \Phi(E)$, then~$T[E]$ is automatically closed (see \cite[Corollary~3.2.5]{caradus}). The \textit{index} of a Fredholm operator~$T \in \Phi(E)$ is defined as
\[
i(T):= \dim \ker T - \dim(E/T[E]).
\]
It is well-known that the Fredholm index
\[
 i \colon \Phi(E) \to \mathbb{Z}; \quad T \mapsto \dim \ker T - \dim( E/T[E])
\]
is a continuous map (see~\cite[Theorem~4.4.1]{caradus}). A useful consequence is that the Fredholm index $i$ is constant on each connected component of~$\Phi(E)$. Indeed, the continuous map~$i$ maps connected subsets of~$\Phi(E)$ to connected subsets of~$\mathbb{Z}$. However, the only connected subsets of~$\mathbb{Z}$ are the singletons.
\medskip

We recall that if the Banach space~$E$ has the PFP, then~$\ME$ is the largest two-sided ideal in~$\Bop(E)$ by Proposition~\ref{prop:PFP-maximal-ideal}, thus
\[
A_E := \Bop(E) /\ME
\]
is a simple Banach algebra. In the positive case,
\[
q \colon \Bop(E) \to \Bop(E) /\ME = A_E
\]
denotes the quotient homomorphism.

\begin{lemma}\label{lemma: image_of_fredholms}
    If the Banach space~$E$ has the PFP, then
    \[
    q[\Phi(E)] \subseteq \inv A_E.
    \]
\end{lemma}

\begin{proof}
    Let~$\Kideal(E)$ denote the closed, two-sided ideal of compact operators on~$E$, and let \newline $\pi \colon \Bop(E) \to \Bop(E)/ \Kideal(E)$ be the quotient homomorphism. As~$\Kideal(E) \subseteq \ME$, there is a continuous, surjective algebra homomorphism
    \[
    \rho \colon \Bop(E) / \Kideal(E) \to \Bop(E) / \ME = A_E
    \]
    with~$\rho \circ \pi = q$. As~$\pi[\Phi(E)]= \inv (\Bop(E)/\Kideal(E))$ by Atkinson's Theorem (see \cite[Theorem~3.2.8]{caradus}), in particular
    \[
    q[\Phi(E)] = \rho[\pi[\Phi(E)]] = \rho[\inv(\Bop(E)/ \Kideal(E))] \subseteq \inv(\Bop(E) / \ME) = \inv A_E,
    \]
    as claimed.
\end{proof}

\begin{corollary}\label{cor: pfp_fredholm}
    If the Banach space~$E$ has the PFP, then~$\Phi(E) \subseteq \Bop(E) \setminus \ME$.
\end{corollary}

\begin{proof}
    Let~$T \in \Phi(E)$. Then~$q(T) \in \inv A_E$ by Lemma~\ref{lemma: image_of_fredholms}, consequently~$q(T) \not = 0$ or equivalently~$T \notin \ME$.
\end{proof}

\begin{lemma}\label{lemma: semi_fredholm_large_beta}
    If the Banach space~$E$ has the PFP, then for each~$n\in \mathbb{N}$ there is a~$L \in \Bop(E) \setminus \ME$ such that~$\dim \ker L =0$ and~$\dim(E/L[E]) \geq n$.
\end{lemma}

\begin{proof}
    Let~$P \in \Bop(E)$ be an idempotent operator of rank~$n \in \mathbb{N}$. As~$P \in \Kideal(E) \subseteq \ME$, we must have~$\Id_E - P \notin \ME$ by the PFP. Let~$U,V \in \Bop(E)$ with~$\Id_E = U(\Id_E - P)V$. Setting~$L:=(\Id_E - P)V$, clearly~$\Id_E=UL$, and hence~$L \notin \ME$. Since~$L$ is injective, $\dim \ker L= 0$ follows. Also,
    \[
    L[E] \subseteq (\Id_E - P)[E] = \ker P
    \]
    follows, and therefore~$E/L[E]$ surjects linearly onto~$E/ \ker P \simeq P[E]$. Consequently, \newline $\dim(E/L[E]) \geq \dim P[E] = n$.
\end{proof}

Recall that an infinite-dimensional Banach space~$E$ is \textit{indecomposable} if there are no closed, infinite-dimensional subspaces~$F,G$ of~$E$ such that~$E= F \oplus G$. In other words, $E$ is indecomposable if and only if~$\ker P$ or~$P[E]$ is finite-dimensional whenever~$P \in \Bop(E)$ is an idempotent operator.

\begin{lemma}\label{lemma: indecomposable_fredholm}
    If the Banach space~$E$ is indecomposable, then
    \[
    \Bop(E) \setminus \ME \subseteq \Phi(E).
    \]
\end{lemma}

\begin{proof}
    Let~$T \in \Bop(E)$ be such that~$\Id_E = UTV$ for some~$U,V \in \Bop(E)$. By Lemma~\ref{lemma: idempotents_from_factorisation}, there exist idempotent operators~$P,Q \in \Bop(E)$ with~$P[E] \simeq E \simeq Q[E]$, $\ker T \subseteq \ker P$ and~$Q[E] \subseteq T[E]$. The latter implies that~$\ker Q \simeq E/Q[E]$ surjects linearly onto~$E/T[E]$. As~$E$ is indecomposable, $\ker P$ and~$\ker Q$ must be finite-dimensional. Consequently,
    \[
    \dim(E/T[E]) \leqslant \dim(E/Q[E]) = \dim \ker Q < \infty,
    \]
    and~$\dim \ker T \leqslant \dim \ker P < \infty$. Thus~$T \in \Phi(E)$, as claimed.
\end{proof}

\begin{theorem}\label{thm:no-indecomposable-PFP}
There is no real or complex indecomposable Banach space with the PFP.
\end{theorem}

\begin{proof}
    Assume towards a contradiction that~$E$ is an indecomposable Banach space with the PFP.
    \smallskip
    
    \textit{Claim~1.} For some $\mathbb{F} \in \{\mathbb{R}, \mathbb{C}, \mathbb{H} \}$, the Banach algebra~$\Bop(E)$ admits an~$\mathbb{F}$-valued character whose kernel is~$\ME$. More precisely, there is a continuous, unital algebra homomorphism~$\varphi \colon \Bop(E) \to \mathbb{F}$ with~$\ker \varphi = \ME$.
    \smallskip
    
    \begin{proof}[Proof of Claim~1.]
        From Lemmas~\ref{lemma: image_of_fredholms} and~\ref{lemma: indecomposable_fredholm} the containment
        \[
        (\Bop(E) / \ME) \setminus \{0\} = q[\Bop(E) \setminus \ME] \subseteq q[\Phi(E)] \subseteq \inv(\Bop(E) / \ME)
        \]
        follows. Consequently~$\Bop(E)/ \ME$ is a real or complex division Banach algebra, hence by the Gelfand--Mazur Theorem (see \cite[Chapter~14, Theorem~7]{bonsall}) there exists a continuous, unital algebra homomorphism~$\varphi \colon \Bop(E) \to \mathbb{F}$ with~$\ker \varphi = \ME$.
    \end{proof}
    \smallskip
    
    \textit{Claim~2.} $\Phi(E) = \Bop(E) \setminus \ME = \Bop(E) \setminus \ker \varphi$.
    
    \begin{proof}[Proof of Claim~2.]
        The identity~$\Phi(E) = \Bop(E) \setminus \ME$ is a straightforward consequence of Corollary~\ref{cor: pfp_fredholm} and Lemma~\ref{lemma: indecomposable_fredholm}. The last identity follows from Claim~1.
    \end{proof}
    \smallskip
    
    \textit{Claim~3.} Every Fredholm operator on~$E$ has zero Fredholm index.
    \begin{proof}[Proof of Claim~3.]
        By Claim~2 and Lemma~\ref{lemma: same_component}, every operator in~$\Phi(E)$ lies in the same connected component as~$\varepsilon \cdot \Id_E$, where $\varepsilon \in \{1, -1 \}$. As the Fredholm index is constant on connected components, every operator in~$\Phi(E)$ has Fredholm index~$i(\varepsilon \cdot \Id_E)=0$.
    \end{proof}
    \smallskip
    
    In view of Lemma~\ref{lemma: semi_fredholm_large_beta}, we can pick~$L \in \Bop(E) \setminus \ME$ such that~$\dim \ker L =0$ and \newline $\dim(E/L[E]) \geq 1$. By Lemma~\ref{lemma: indecomposable_fredholm} we have~$L \in \Phi(E)$, thus
    \[
    i(L)= \dim \ker L - \dim(E/L[E]) \leq -1.
    \]
    By Claim~3 however we have~$i(L)=0$, which is a contradiction. Thus there is no indecomposable Banach space with the PFP.
\end{proof}

\begin{corollary}\label{cor:indecomposable-examples-no-PFP}
The following Banach spaces do not have the PFP:
\begin{enumerate}
\item hereditarily indecomposable Banach spaces;
\item Tarbard's indecomposable Bourgain--Delbaen space \(X_\infty\)
from \cite[Theorem~4.1.1]{Tarbard};
\item the spaces \(C(K)\), where \(K\) is a connected Koszmider compactum
as in \cite{Koszmider2004}.
\item the indecomposable, prime Banach space of Gowers and Maurey from \cite{GowersMaurey}.
\end{enumerate}
\end{corollary}

\begin{proof}
Each space listed above is indecomposable: this is immediate in the first
case, part of Tarbard's construction in the second, and follows from
Koszmider's few-operators result in the third.  The conclusion is therefore
Theorem~\ref{thm:no-indecomposable-PFP}.
\end{proof}

It should be noted that Corollary \ref{cor:indecomposable-examples-no-PFP}(4) reaches the same conclusion as \cite[Example 7.1]{AcuavivaKania2026} but via a rather different route.

\section{PFP and ring-theoretic infiniteness}

We first prove the Dedekind-infiniteness observation.

\begin{lemma}\label{lem:PFP-Dedekind-infinite}
If the infinite-dimensional Banach space $E$ has the PFP, then $\Bop(E)$ is Dedekind-infinite.
\end{lemma}

\begin{proof}
Choose~$x\in E$ and~$f\in E^*$ with~$\norm x=\norm f=1$ and~$f(x)=1$.  Put
\(
  P=x\otimes f,
\)
so that~$P$ is a rank-one idempotent on~$E$. By the PFP, there are~$U,V\in\Bop(E)$ such that
\[
  \Id_E=UPV \qquad\hbox{or} \qquad \Id_E=U(\Id_E-P)V.
\]
The first alternative is impossible, because~$UPV$ has rank at most one whereas~$E$ is infinite-dimensional.  Hence
\[
  \Id_E=U(\Id_E-P)V.
\]
Define
\[
  \widehat{P}=VU(\Id_E-P).
\]
By Lemma~\ref{lemma: idempotents_from_factorisation}, $\widehat{P}\in\Bop(E)$ is an idempotent operator with~$\widehat{P}[E] \simeq E$ and~$\ker(\Id_E -P) \subseteq \ker \widehat{P}$. On the one hand~$\widehat{P}[E] \simeq E$ implies~$\widehat{P} \sim \Id_E$. On the other hand~$Px=x$ by construction, therefore~$x \in \ker(\Id_E -P) \subseteq \ker \widehat{P}$. Hence~$\widehat{P} \neq \Id_E$, proving that~$\Bop(E)$ is Dedekind-infinite.
\end{proof}

The preceding lemma cannot be strengthened to proper infiniteness.

\begin{proposition}\label{prop:C-omega-one-not-properly-infinite}
Let $E=C[0,\omega_1]$.  Then $E$ has the PFP, but $\Bop(E)$ is not properly infinite.
\end{proposition}

\begin{proof}
Loy and Willis constructed a codimension-one ideal in $\Bop(C[0,\omega_1])$ \cite{LoyWillis1989}.  The third-named author and Laustsen showed that this ideal is the unique maximal ideal and identified it, in coordinate-free terms, as the set of operators through which the identity does not factor \cite{KaniaLaustsen2012}.  Thus $\ME$ is a codimension-one two-sided ideal in $\Bop(E)$, and $E$ has the PFP by Proposition~\ref{prop:PFP-maximal-ideal}.  Moreover,
\[
  \Bop(E)/\ME\simeq\C.
\]
If $\Bop(E)$ were properly infinite, then its unital quotient $\C$ would be properly infinite.  This is impossible.  Hence $\Bop(E)$ is not properly infinite.
\end{proof}

We next pass from $\Bop(E)$ to the quotient by $\ME$.

\begin{proposition}\label{prop:quotient-pure-infinite}
Let $E$ be a complex Banach space with the PFP, and set
\[
  A_E=\Bop(E)/\ME.
\]
Then every non-zero element of $A_E$ has the identity factoring through it.  Consequently, $A_E$ is purely infinite if and only if $A_E\not\simeq\C$.
\end{proposition}

\begin{proof}
By Proposition~\ref{prop:PFP-maximal-ideal}, $\ME$ is a closed maximal ideal of $\Bop(E)$, and therefore $A_E$ is a unital complex Banach algebra.  Let
\[
  \pi\colon\Bop(E)\to A_E
\]
be the quotient homomorphism, and take $0\neq a\in A_E$.  Choose $T\in\Bop(E)$ such that $a=\pi(T)$.  Since $a\neq0$, we have $T\notin\ME$.  Hence there are $U,V\in\Bop(E)$ such that
\[
  UTV=\Id_E.
\]
Applying $\pi$ gives
\[
  \pi(U)a\pi(V)=1_{A_E}.
\]
Thus, the identity of $A_E$ factors through every non-zero element of $A_E$.  The final assertion follows immediately from Definition~\ref{def:pure-infinite}.

\end{proof}

\begin{remark}\label{rem:K-theory}
For a non-scalar complex PFP space, Proposition~\ref{prop:quotient-pure-infinite}
identifies \(A_E\) as a simple purely infinite Banach algebra.  This places the
quotients considered here close to questions in Banach-algebra \(K\)-theory.  For
example, the third-named author, Koszmider and Laustsen proved that
\(K_0(\Bop(C[0,\omega_1]))\simeq\mathbb Z\) \cite{KaniaKoszmiderLaustsen2015}.
The scalar quotient \(\Bop(C[0,\omega_1])/\ME\simeq\C\) shows that this example
sits on the boundary of the present pure-infiniteness problem, but it suggests
that the \(K\)-theory of the non-scalar quotients \(A_E\) may be worth studying
separately.
\end{remark}

Proposition~\ref{prop:quotient-pure-infinite} reduces pure infiniteness of
\(\Bop(E)/\ME\) to excluding the scalar case.  A~standard geometric way to do
this is to require \(E\) to contain two complemented copies of itself inside a
complemented subspace.  Under this hypothesis \(\Bop(E)\) is properly infinite,
and hence it cannot have a non-zero finite-dimensional quotient: proper
infiniteness passes to unital quotients, whereas finite-dimensional unital
algebras are Dedekind-finite.  The following observation is essentially the
combination of \cite[Proposition~1.9 and Lemma~1.11]{Laustsen2003}; we include
the short proof for convenience.

\begin{proposition}\label{prop:no-finite-dimensional-quotient-square}
Let $E$ be a Banach space which contains a complemented subspace isomorphic to $E\oplus E$.  Then $\Bop(E)$ has no non-zero finite-dimensional quotient. In particular, if $E$ has the PFP, then $\Bop(E)/\ME\not\simeq\C$.
\end{proposition}

\begin{proof}
Let $F$ be a complemented subspace of $E$ with $F\simeq E\oplus E$.  The two coordinate copies of $E$ inside $F$ give orthogonal idempotents $p,q\in\Bop(E)$ whose ranges are complemented and isomorphic to $E$.  Hence
\[
  p\sim\Id_E\sim q,
\]
so $\Bop(E)$ is properly infinite.

Suppose that $\theta\colon\Bop(E)\to B$ is a non-zero unital homomorphism onto a finite-dimensional algebra $B$.  Proper infiniteness passes to unital quotients, so $B$ is properly infinite and hence Dedekind-infinite.  But every finite-dimensional unital algebra is Dedekind-finite: if $ab=1$ in such an algebra, then left multiplication by $a$ has a right inverse, hence is invertible as a linear map, and therefore $ba=1$.  This contradiction proves the first assertion.  The final assertion follows because $\Bop(E)/\ME\simeq\C$ would be a non-zero finite-dimensional quotient.
\end{proof}

Propositions~\ref{prop:quotient-pure-infinite} and~\ref{prop:no-finite-dimensional-quotient-square} thus amount to the following result:

\begin{corollary}
    Let~$E$ be a complex Banach space which contains a complemented subspace isomorphic to~$E\oplus E$. If~$E$ has the PFP, then~$A_E = \Bop(E) / \ME$ is purely infinite.
\end{corollary}

\section{Ultrapowers of the quotient}

Let $A$ be a unital Banach algebra.  For $a\in A\setminus\{0\}$ define
\begin{equation}\label{eq:cpi}
  \cpi^A(a)=\inf\{\norm b\norm c:\ b,c\in A,\ bac=1_A\},
\end{equation}
with the convention that the infimum is $\infty$ if no such $b,c$ exist.  We also set $\cpi^A(0)=\infty$ when convenient.  The quantity is homogeneous:
\begin{equation}\label{eq:cpi-homogeneity}
  \cpi^A(\lambda a)=|\lambda|^{-1}\cpi^A(a)\qquad(\lambda\in\C\setminus\{0\}).
\end{equation}

\begin{remark}[The geometry behind~$\cpi$]\label{rem:geometry-cpi}
In quotients of operator algebras, the constants in \eqref{eq:cpi} measure, in a precise way, how efficiently one finds complemented copies of the underlying space inside ranges of operators.  Suppose, for instance, that~$T\in\Bop(E)$ and that
\[
  UTV=\Id_E.
\]
Then
\(
  Q=TVU
\)
is a projection on~$E$ with~$Q[E] \subseteq T[E]$ by Lemma~\ref{lemma: idempotents_from_factorisation}.  Moreover~$TV\colon E\to TV[E]$ is an isomorphism with inverse~$U|_{TV[E]}$. Thus a factorisation of the identity through~$T$ gives a complemented copy of~$E$ inside~$T[E]$, with projection norm and isomorphism constants controlled by the norms of~$U$, $V$ and~$T$.  Conversely, whenever a concrete Banach-space argument produces well-complemented copies of~$E$ inside the range of each operator representing a quotient element, it gives estimates for~$\cpi$.

This is the mechanism behind the constant-one result for the classical sequence spaces below.  It also explains why the non-scalar negative problem is delicate: one needs a PFP space for which the qualitative dichotomy holds, while the constants for the complemented copies forced by one side of the dichotomy can nevertheless escape to infinity.
\end{remark}

For an ultrafilter~$\mathcal U$, we write~$(A)^{\mathcal U}$ for the Banach-algebra ultrapower of~$A$.  Daws and the second-named author obtained the following quantitative characterisation.

\begin{theorem}[Daws and the second-named author]\label{thm:DH-criterion}
Let~$A$ be a unital Banach algebra with~$A\not\simeq\C$, and let~$\mathcal U$ be a countably incomplete ultrafilter. Then~$(A)^{\mathcal U}$ is purely infinite if and only if
\begin{equation}\label{eq:DH-uniform-bound}
  \sup\{\cpi^A(a):\ a\in A,\ \norm a=1\}<\infty.
\end{equation}
\end{theorem}

\begin{proof}
This is~\cite[Proposition~2.2]{DawsHorvath2022}.
\end{proof}

The UPFP gives a uniform estimate, but only after allowing the complementary element~$1-a$.

\begin{proposition}\label{prop:UPFP-quotient-dichotomy}
Let~$E$ be a complex Banach space with the~$C$-PFP. Put
\[
  A_E=\Bop(E)/\ME.
\]
Then, for every~$a\in A_E$,
\begin{equation}\label{eq:quotient-dichotomy}
  \cpi^{A_E}(a)\le C
  \qquad\hbox{or}\qquad
  \cpi^{A_E}(1_{A_E}-a)\le C.
\end{equation}
\end{proposition}

\begin{proof}
Let~$\pi\colon\Bop(E)\to A_E$ be the quotient map, and choose~$T\in\Bop(E)$ with~$\pi(T)=a$. By the~$C$-PFP, there are~$U,V\in\Bop(E)$ with~$\norm U\norm V\le C$ such that either
\[
  UTV=\Id_E
  \qquad\hbox{or}\qquad
  U(\Id_E-T)V=\Id_E.
\]  Applying~$\pi$ gives either
\[
  \pi(U)a\pi(V)=1_{A_E} \text{ or }\pi(U)(1_{A_E}-a)\pi(V)=1_{A_E}.
\]
Since~$\norm{\pi(U)}\norm{\pi(V)}\le\norm U\norm V\le C$, the conclusion follows.
\end{proof}

\begin{remark}\label{rem:gap}
The estimate \eqref{eq:quotient-dichotomy} is not the same as the uniform bound \eqref{eq:DH-uniform-bound}.  The latter requires a bound for~$\cpi^{A_E}(a)$ on the entire unit sphere. The former says only that, for each~$a$, either~$a$ or~$1-a$ has a controlled factorisation. Thus the UPFP does not by itself give a~formal proof that~$(A_E)^{\mathcal U}$ is purely infinite.
\end{remark}

The work of Daws and the second-named author also gives an obstruction to pure infiniteness of ultrapowers.

\begin{proposition}\label{prop:homomorphic-obstruction}
Let~$A$ be a unital Banach algebra.  Suppose that there are a Banach algebra~$B$ and a non-zero continuous algebra homomorphism
\[
  \theta\colon A\to B
\]
which is not bounded below. Then~$(A)^{\mathcal U}$ is not simple for every countably incomplete ultrafilter~$\mathcal U$.  In particular, $(A)^{\mathcal U}$ is not purely infinite.
\end{proposition}

\begin{proof}
This is the contrapositive of~\cite[Proposition~2.8]{DawsHorvath2022}: if a countably incomplete ultrapower of~$A$ is simple, then every non-zero continuous algebra homomorphism from~$A$ into a Banach algebra is bounded below.  Purely infinite Banach algebras are simple, so the final assertion follows.
\end{proof}

Applied to~$A=A_E$, Proposition~\ref{prop:homomorphic-obstruction} gives a concrete strategy for producing non-purely-infinite ultrapowers of the quotient: find a non-zero continuous homomorphism out of~$A_E$ which fails to be bounded below.

\begin{remark}\label{rem:homomorphic-rigidity}
This obstruction is stronger than may first appear in the present setting.  If~$E$ has the PFP, then Proposition~\ref{prop:PFP-maximal-ideal} makes~$A_E$ a simple Banach algebra. Consequently every non-zero continuous homomorphism~$A_E\to B$ is automatically injective.  Proposition~\ref{prop:homomorphic-obstruction} is therefore asking for an injective continuous representation of this Calkin-type algebra which is not bounded below, or equivalently for a strong failure of uniqueness of the quotient norm.  This rigidity is one reason that the non-scalar questions below are not merely a matter of finding more examples with the UPFP.
\end{remark}

\section{Examples, boundary cases and test problems}

We now organise the examples according to the distinction between scalar boundary cases and the genuinely non-scalar ultrapower problem.

\begin{proposition}\label{prop:scalar-boundary-PFP}
The non-scalar hypothesis in Theorem~\ref{thm:intro-main}\textup{(iii)} is necessary.  There is an infinite-dimensional complex Banach space~$E$ with the PFP such that
\[
  \bigl(\Bop(E)/\ME\bigr)^{\mathcal U}
\]
is not purely infinite for every ultrafilter~$\mathcal U$.
\end{proposition}

\begin{proof}
Take~$E=C[0,\omega_1]$.  By Proposition~\ref{prop:C-omega-one-not-properly-infinite}, $E$ has the PFP and
\[
  \Bop(E)/\ME\simeq\C.
\]
The ultrapower of the one-dimensional Banach algebra~$\C$ is again isometrically isomorphic to $\C$, via the ultralimit map on bounded scalar families.  Since $\C$ is not purely infinite by Definition~\ref{def:pure-infinite}, the result follows.
\end{proof}

\begin{remark}[A finite-dimensional scalar UPFP boundary]\label{rem:scalar-boundary-UPFP}
There is also a degenerate finite-dimensional scalar UPFP boundary case.  Take $E=\C$.  Then $\Bop(E)\simeq\C$ and $\ME=\{0\}$, so $\Bop(E)/\ME\simeq\C$.  Moreover, $\C$ has the $2$-PFP: for each $\lambda\in\C$, either $|\lambda|\ge1/2$ or $|1-\lambda|\ge1/2$, and hence the identity factors through multiplication by $\lambda$ or by $1-\lambda$ with product of norms at most $2$.  Thus $E$ has the UPFP, while the ultrapower of the quotient is again $\C$ and is therefore not purely infinite.
\end{remark}

The preceding examples are useful only as boundary cases: they show that the scalar alternative must be excluded.  They do not address the problem suggested by Theorem~\ref{thm:DH-criterion}, where the original quotient is required to be purely infinite.

\begin{question}\label{q:non-scalar-PFP}
Does there exist an infinite-dimensional complex Banach space~$E$ with the PFP such that
\[
  A_E=\Bop(E)/\ME\not\simeq\C
\]
and, for some countably incomplete ultrafilter~$\mathcal U$, the ultrapower~$(A_E)^{\mathcal U}$ is not purely infinite?
\end{question}

Equivalently, by Theorem~\ref{thm:DH-criterion}, Question~\ref{q:non-scalar-PFP} asks whether there is such an $E$ for which $\cpi^{A_E}$ is unbounded on the unit sphere of $A_E$.

\begin{question}\label{q:non-scalar-UPFP}
Does there exist an infinite-dimensional complex Banach space~$E$ with the UPFP such that
\[
  A_E=\Bop(E)/\ME\not\simeq\C
\]
and, for some countably incomplete ultrafilter~$\mathcal U$, the ultrapower~$(A_E)^{\mathcal U}$ is not purely infinite?
\end{question}

For Question~\ref{q:non-scalar-UPFP}, Proposition~\ref{prop:UPFP-quotient-dichotomy} shows exactly what is missing: one has a uniform factorisation for $a$ or for $1-a$, but the criterion of Daws and the second-named author requires a uniform factorisation for $a$ itself.  Proposition~\ref{prop:homomorphic-obstruction} gives a possible route to a positive answer: it would suffice to find a non-zero continuous homomorphism
\[
  \theta\colon A_E\to B
\]
which is not bounded below.

\begin{remark}\label{rem:nonscalar-UPFP-tests}
The non-scalar UPFP problem is not vacuous.  Recent work of the first- and third-named authors develops general UPFP machinery for symmetric and iterated direct sums \cite{AcuavivaKania2026}, and a subsequent preprint of the first-named author proves that $\ell_p(C[0,1])$ has the UPFP for $1\le p\le\infty$ \cite{Acuaviva2026LpCK}.  These spaces are isomorphic to their squares, and so Proposition~\ref{prop:no-finite-dimensional-quotient-square} shows that the corresponding quotient $A_E$ is not scalar.  Thus, they are natural non-scalar test cases for Question~\ref{q:non-scalar-UPFP}.  What is not obtained merely from the UPFP is a uniform bound for $\cpi^{A_E}(a)$ on the quotient unit sphere; this would require a one-sided quantitative estimate for $a$, not only the dichotomy between $a$ and $1-a$.
\end{remark}

\begin{proposition}\label{prop:square-zero-obstruction}
Let $E$ be a Banach space, and suppose that there are closed two-sided ideals $\ideal N\subsetneq \ideal I\subseteq \Bop(E)$ such that $\ideal I$ has codimension one in $\Bop(E)$ and $\ideal I^2\subseteq\ideal N$.  Then $E$ does not have the UPFP.
\end{proposition}

\begin{proof}
Put~$A=\Bop(E)$ and suppose, towards a contradiction, that~$E$ has the~$C$-PFP for some~$C\ge 1$.  Let~$q\colon A\to A/\ideal N$ be the quotient homomorphism and set~$B=q[\ideal I]$. Then~$B$ is a non-zero square-zero ideal of~$A/\ideal N$. Since~$\ideal I$ has codimension one, there is a character~$\chi\colon A\to\K$ with kernel~$\ideal I$ and~$\chi(\Id_E)=1$; as~$\ideal N\subseteq\ideal I$, it descends to a continuous character~$\widetilde\chi$ on~$A/\ideal N$ with kernel~$B$.

Choose~$S\in\ideal I$ with~$q(S)\ne0$. Fix~$t\in\K$, $t\ne0$, and apply the~$C$-PFP to~$T_t=\Id_E-tS$. The alternative in which the identity factors through~$\Id_E-T_t=tS$ is impossible, because~$tS$ lies in the proper ideal~$\ideal I$. Hence there are~$U_t,V_t\in A$ such that~$U_t(\Id_E-tS)V_t=\Id_E$ and~$\norm{U_t}\norm{V_t}\le C$.

Pass to~$A/\ideal N$. Write~$u=q(U_t)$, $v=q(V_t)$, $b=tq(S)$, $\alpha=\widetilde\chi(u)$ and~$\beta=\widetilde\chi(v)$. Clearly~$u(1-b)v=1$. We \textit{claim} that~$\alpha\beta=1$. Indeed, $U_tSV_t \in \ideal I$ as~$S \in \ideal I$, and therefore
\[
\widetilde\chi(ubv) =t\widetilde\chi(q(U_tSV_t)) = t \chi(U_tSV_t) = 0,
\]
which together with
\[
1=\chi(\Id_E) = \widetilde\chi(1) = \widetilde\chi(u(1-b)v) = \widetilde\chi(u) \widetilde\chi(v) - \widetilde\chi(ubv) = \alpha \beta - \widetilde\chi(ubv)
\]
yields the claim.

Decompose~$u=\alpha1+u_0$ and~$v=\beta1+v_0$, with~$u_0,v_0\in B$. Since~$B^2=0$ and~$b\in B$, by the Claim the equation~$u(1-b)v=1$ reduces to~$b=\alpha v_0+\beta u_0$. Therefore
%
%
\[
  |t|\norm{q(S)}=\norm b\le |\alpha|\norm{v_0}+|\beta|\norm{u_0}
  \le 2\norm{\widetilde\chi}(\norm{\widetilde\chi}+1)\,\norm u\norm v+2\norm1
  \le 2C\norm{\widetilde\chi}(\norm{\widetilde\chi}+1)+2\norm1.
\]
The right-hand side is independent of~$t$, a contradiction as~$|t|\to\infty$. Hence~$E$ cannot have the UPFP.
\end{proof}

\begin{theorem}\label{thm:Read-not-UPFP}
Read's space~$\ER$ does not have the UPFP.
\end{theorem}

\begin{proof}
Read constructs a closed codimension-one ideal~$\JR$ in~$\Bop(\ER)$ with the following stronger property.  Let
\[
   \pi\colon \Bop(\ER)\longrightarrow \Bop(\ER)/\Wideal(\ER)
\]
be the quotient map.  Then~$\pi[\JR]$ is a non-zero closed square-zero ideal; equivalently,
\[
  \Wideal(\ER)\subsetneq \JR
  \qquad\hbox{and}\qquad
  \JR^2\subseteq \Wideal(\ER).
\]
Read proves that the corresponding ideal~$I$ in~$\Bop(\ER)/\Wideal(\ER)$ is closed, codimension one and satisfies~$I^2=0$; its inverse image~$\JR=\pi^{-1}[I]$ is therefore a closed codimension-one ideal of~$\Bop(\ER)$, see~\cite[Section~4]{Read1989}.

Apply Proposition~\ref{prop:square-zero-obstruction} with
\[
   \ideal I=\JR,
   \qquad
   \ideal N=\Wideal(\ER).
\]
This proves that~$\ER$ does not have the UPFP.
\end{proof}

\begin{remark}\label{rem:Read-obstruction-geometric}
The obstruction in Proposition~\ref{prop:square-zero-obstruction} is the quantitative core of the geometric difficulty suggested by Read's construction.  Uniform primary factorisation would force uniformly bounded factorisations of $\Id_E$ through every scalar perturbation $\Id_E-tS$, where $S$ lies in the codimension-one ideal.  Passing to the square-zero quotient linearises these factorisations and shows that the vector $tq(S)$ would have to remain uniformly bounded.  Thus the failure of UPFP is already visible in the weak Calkin algebra, without constructing an~explicit weakly compact perturbation of the identity.
\end{remark}

\begin{lemma}\label{lem:Read-forward-inclusion}
Let \(\psi\colon\Bop(\ER)\to\widetilde{\ell}_2\) be the
Laustsen--Skillicorn quotient homomorphism, and put \(\IR=\psi^{-1}[\ell_2]\).
Then
\[
   \IR\subseteq \EuScript{M}_{\ER}.
\]
Equivalently, no operator in Read's codimension-one ideal \(\IR\) factors the
identity.
\end{lemma}

\begin{proof}
Let \(T\in\IR\), so \(\psi(T)\in\ell_2\).  Suppose that \(UTV=\Id_{\ER}\) for
some \(U,V\in\Bop(\ER)\).  Write \(\psi(U)=\eta+\alpha1\) and
\(\psi(V)=\zeta+\beta1\), with \(\eta,\zeta\in\ell_2\).  Since \(\ell_2\) has
zero product in its unitisation, \(\psi(U)\psi(T)\psi(V)=\alpha\beta\psi(T)\)
has scalar part zero, whereas \(\psi(\Id_{\ER})=1\) has scalar part one.  This
contradiction proves the claim.
\end{proof}

\begin{proposition}\label{prop:Read-reduction}
Let $\ER$ be Read's Banach space, and let
\[
  \psi\colon\Bop(\ER)\longrightarrow\widetilde{\ell}_2
\]
be the Laustsen--Skillicorn quotient homomorphism onto the unitisation of $\ell_2$ endowed with zero product \cite{LaustsenSkillicorn2017}.  Thus
\[
  \ker\psi=\Wideal(\ER),
\]
and the short exact sequence
\[
  0\longrightarrow \Wideal(\ER)\longrightarrow \Bop(\ER)\stackrel{\psi}{\longrightarrow}\widetilde{\ell}_2\longrightarrow0
\]
splits strongly.  Put
\[
  \IR=\psi^{-1}[\ell_2],
\]
the codimension-one ideal obtained by pulling back the codimension-one ideal $\ell_2$ of $\widetilde{\ell}_2$; this is the same ideal as $\JR$ above.

Then $\ER$ has the PFP if and only if, for every $W\in\Wideal(\ER)$, the identity on $\ER$ factors through $\Id_{\ER}-W$.  If these equivalent conditions hold, then
\[
  \EuScript{M}_{\ER}=\IR
  \qquad\hbox{and}\qquad
  \Bop(\ER)/\EuScript{M}_{\ER}\simeq\C.
\]
\end{proposition}

\begin{proof}
The structural assertions about $\psi$, its kernel and the strongly split exact sequence are due to Laustsen and Skillicorn, strengthening Read's original analysis \cite{Read1989,LaustsenSkillicorn2017}.  We prove the stated equivalence.

Assume first that $\ER$ has the PFP.  By Proposition~\ref{prop:PFP-maximal-ideal}, $\EuScript{M}_{\ER}$ is the largest proper ideal of $\Bop(\ER)$.  Lemma~\ref{lem:Read-forward-inclusion} gives $\IR\subseteq\EuScript{M}_{\ER}$, and since $\IR$ has codimension one, we get $\IR=\EuScript{M}_{\ER}$.  If $W\in\Wideal(\ER)$, then
\[
  \psi(\Id_{\ER}-W)=1,
\]
so $\Id_{\ER}-W\notin\IR=\EuScript{M}_{\ER}$.  Hence the identity factors through $\Id_{\ER}-W$.

Conversely, suppose that $\Id_{\ER}$ factors through $\Id_{\ER}-W$ for every $W\in\Wideal(\ER)$.  Let $T\in\Bop(\ER)$, and write
\[
  \psi(T)=\xi+\lambda 1
  \qquad(\xi\in\ell_2,\ \lambda\in\C).
\]
If $\lambda=0$, then $\psi(\Id_{\ER}-T)=-\xi+1$, so it is enough to prove that the identity factors through every operator whose image under $\psi$ has non-zero scalar part.  Assume therefore that $\lambda\ne0$.  Since $\ell_2$ has zero product in $\widetilde{\ell}_2$, the element $\xi+\lambda1$ is invertible, with inverse
\[
  \lambda^{-1}1-\lambda^{-2}\xi.
\]
By surjectivity of $\psi$, choose $S\in\Bop(\ER)$ such that
\[
  \psi(S)=\lambda^{-1}1-\lambda^{-2}\xi.
\]
Then $\psi(ST)=1$, and therefore $ST=\Id_{\ER}-W$ for some $W\in\Wideal(\ER)$.  By the hypothesis, there are $U,V\in\Bop(\ER)$ such that
\[
  U(\Id_{\ER}-W)V=\Id_{\ER}.
\]
Consequently
\[
  USTV=\Id_{\ER},
\]
so the identity factors through $T$.  Thus, for every $T$, the identity factors through $T$ or through $\Id_{\ER}-T$, and $\ER$ has the PFP.

If these conditions hold, the preceding argument also shows that every operator outside $\IR$ factors the identity.  No operator in $\IR$ can factor the identity by Lemma~\ref{lem:Read-forward-inclusion}.  Hence $\EuScript{M}_{\ER}=\IR$, and the quotient is $\widetilde{\ell}_2/\ell_2\simeq\C$.
\end{proof}

\begin{question}\label{q:Read-PFP}
Does Read's space $\ER$ have the PFP?  Equivalently, does every weakly compact perturbation $\Id_{\ER}-W$ fix a complemented copy of $\ER$ through which the identity factors?
\end{question}

\begin{remark}\label{rem:Read-consequence}
The known weak Calkin algebra structure does not appear to settle Question~\ref{q:Read-PFP}; it reduces the question to the complemented-copy assertion in Proposition~\ref{prop:Read-reduction}.  Theorem~\ref{thm:Read-not-UPFP} shows, however, that any affirmative answer would necessarily be non-uniform.  In particular, it would answer the question of Acuaviva and the third-named author whether there is a Banach space with the PFP but without the UPFP \cite[Question~7.4]{AcuavivaKania2026}.  Such a positive answer would strengthen the scalar boundary examples by adding a space $E$ for which $\Bop(E)$ admits discontinuous homomorphisms into Banach algebras \cite{Read1989}.  It would not answer Question~\ref{q:non-scalar-PFP} or Question~\ref{q:non-scalar-UPFP}, since Proposition~\ref{prop:Read-reduction} would force the quotient $\Bop(\ER)/\EuScript{M}_{\ER}$ to be scalar.
\end{remark}

\begin{proposition}\label{prop:classical-sequence-spaces-positive}
Let \(\Gamma\) be an infinite set, let \(E\) be either \(c_0(\Gamma)\) or
\(\ell_p(\Gamma)\) for some \(1\le p<\infty\), and put \(\kappa=|\Gamma|\).
Then
\[
  \ME=\Kideal_\kappa(E),
\]
where \(\Kideal_\kappa(E)\) denotes the ideal of \(\kappa\)-compact operators
on \(E\).  Moreover,
\[
  \cpi^{\Bop(E)/\Kideal_\kappa(E)}(a)=1
  \qquad (\norm a=1).
\]
Consequently, for every countably incomplete ultrafilter \(\mathcal U\),
\[
  \bigl(\Bop(E)/\ME\bigr)^{\mathcal U}
\]
is purely infinite.  In particular, these classical sequence spaces do not solve
Question~\ref{q:non-scalar-PFP} or Question~\ref{q:non-scalar-UPFP}.
\end{proposition}

\begin{proof}
Let \(\pi\colon\Bop(E)\to \Bop(E)/\Kideal_\kappa(E)\) be the quotient map.
For these spaces, Daws's description of the closed ideals gives that
\(\Kideal_\kappa(E)\) is the largest proper closed ideal of \(\Bop(E)\)
\cite[Theorem~7.4]{Daws2006}; here \(\Kideal_\kappa(E)\) is understood in the
usual ``fewer than \(\kappa\) coordinates'' sense.  We shall use the quantitative
form of the block-basis extraction contained in the proof of that theorem.  Namely,
if \(T\in\Bop(E)\) and \(0<\delta<\norm{\pi(T)}\), then there are a subset
\(K\subseteq\Gamma\) with \(|K|=\kappa\), norm-one vectors
\((x_\gamma)_{\gamma\in K}\) with pairwise disjoint supports, and finitely
supported functionals \((f_\gamma)_{\gamma\in K}\subseteq E^*\) with pairwise
disjoint supports such that \(\norm{f_\gamma}\le\delta^{-1}\) and
\(f_\gamma(Tx_\eta)=\delta_{\gamma\eta}\) for \(\gamma,\eta\in K\).

For countable \(\Gamma\) this is the usual gliding-hump selection.  In the
uncountable case, the point is that the quotient-norm inequality is quantitative:
after deleting any set of coordinates of cardinality less than \(\kappa\), the
remaining part of \(T\) still has norm exceeding \(\delta\).  At each step one
therefore chooses a unit vector supported outside the previously used coordinates
so that \(Tx_\gamma\) has a coordinate block of norm greater than \(\delta\), and
then chooses a finitely supported norming functional of norm at most
\(\delta^{-1}\).  Daws' \(\Delta\)-system and cardinal pigeonhole thinning then
leave \(\kappa\)-many indices for which both the vector supports and the functional
supports are pairwise disjoint, while the estimates
\(\norm{f_\gamma}\le\delta^{-1}\) and
\(f_\gamma(Tx_\eta)=\delta_{\gamma\eta}\) are preserved.

Write \(E_K\) for \(c_0(K)\) in the \(c_0\)-case and for \(\ell_p(K)\) in the
\(\ell_p\)-case.  Define \(V\colon E_K\to E\) by \(Ve_\gamma=x_\gamma\), and
\(R\colon E\to E_K\) by \(Ry=(f_\gamma(y))_{\gamma\in K}\).  Disjointness of
supports gives \(\norm V=1\) and \(\norm R\le\delta^{-1}\); in the
\(\ell_p\)-case this is the usual \(\ell_p\)-sum estimate over the disjoint
supports, while in the \(c_0\)-case the finite disjoint supports imply
\((f_\gamma(y))\in c_0(K)\), because a vector in \(c_0(\Gamma)\) is large on
only finitely many coordinates.  The biorthogonality above gives \(RTV=\Id_{E_K}\).  Since
\(|K|=|\Gamma|\), we identify \(E_K\) isometrically with \(E\), and hence
\(\Id_E\) factors through \(T\) with product of norms at most \(\delta^{-1}\).
Thus \(T\notin\ME\) whenever \(T\notin\Kideal_\kappa(E)\).  The converse is
immediate because \(\Kideal_\kappa(E)\) is a proper operator ideal and
\(\Id_E\notin\Kideal_\kappa(E)\), so \(\ME=\Kideal_\kappa(E)\).

Now take \(a=\pi(T)\) with \(\norm a=1\).  The preceding paragraph, applied
with arbitrary \(0<\delta<1\), gives \(b,c\in\Bop(E)/\Kideal_\kappa(E)\) such
that \(bac=1\) and \(\norm b\norm c\le\delta^{-1}\).  Letting
\(\delta\uparrow1\) yields \(\cpi(a)\le1\).  The reverse inequality follows
from \(1=bac\), \(\norm a=1\), and \(\norm 1=1\) in every unital quotient by a
proper closed ideal.  Hence \(\cpi(a)=1\) on the quotient unit sphere, and
Theorem~\ref{thm:DH-criterion} gives pure infiniteness of every countably
incomplete ultrapower.

Finally, each of the spaces under consideration is isomorphic to its square, so
Proposition~\ref{prop:no-finite-dimensional-quotient-square} shows that the
quotient is non-scalar.  The same estimates also give the UPFP: for every
\(T\in\Bop(E)\), either \(\norm{\pi(T)}\ge1/2\) or
\(\norm{1-\pi(T)}\ge1/2\), and then the identity factors through \(T\) or
\(\Id_E-T\) with a uniform bound.  Thus these spaces lie on the positive
non-scalar UPFP side of the problem.
\end{proof}

We finish with an infinite-dimensional scalar UPFP boundary case.

Let $J$ denote the complex James space, with its usual basis $(e_n)_{n=1}^\infty$, and put
\[
  s_n=e_1+\ldots+e_n\qquad(n\in\N).
\]
The sequence $(s_n)_{n=1}^\infty$ is the spreading basis used by Andrew; it converges in the weak-star topology of $J^{**}$ to the element denoted below by ${\bf 1}$.  We identify $J^{**}$ with $J\oplus\C{\bf 1}$ in the usual way.

\begin{lemma}[Andrew's theorem, with constants tracked]\label{lem:Andrew-uniform}
For every $\delta>0$ there is a constant $C_\delta$ with the following property.  Let $(z_n)_{n=1}^\infty$ be a bounded sequence in $J$ which converges in the weak-star topology of $J^{**}$ to
\[
  z^{**}=\lambda {\bf 1}+y,
  \qquad y\in J,
\]
where $|\lambda|\ge\delta$.  Then there are integers
\(
  n_1<n_2<\ldots
\)
and an operator $R\in\Bop(J)$ such that
\[
  Rz_{n_k}=s_k\qquad(k\in\N)
\]
and $\norm R\le C_\delta$.
\end{lemma}

\begin{proof}
Andrew proves the corresponding qualitative selection theorem in \cite[Theorem~2.1 and Proposition~2.7]{Andrew1981}.  We recall the part of the proof where the constants enter.

It is enough to prove the result for $\lambda=1$, since replacing $z_n$ by
$\lambda^{-1}z_n$ changes the final bound by at most $\delta^{-1}$.  Andrew
uses the coordinatewise multiplication on $J^{**}$; the required algebraic
background is due to Andrew and Green \cite{AndrewGreen1980}.  In particular
$J^{**}$ is identified with a~multiplier algebra acting on $J$, and
multiplication by an element of $J$ has norm controlled by its James norm.

Let $P_N$ be the $N$th basis projection and denote by $D_N$ the coordinate-deletion
operator $D_N(\sum a_i e_i)=\sum a_{i+N}e_i$.  This operator is a contraction:
every difference-sum which occurs in the James norm of $D_Nx$ is a difference-sum
of $x$ using coordinates beyond the first $N$.  Choose $N$ so large that
$y_0=D_N(\Id-P_N)y$ has multiplier norm at most $1/4$, and put
$\widetilde z_n=D_Nz_n$.  Then $\widetilde z_n\to {\bf 1}+y_0$ weak-star in
$J^{**}$.  Thus discarding the first $N$ coordinates has introduced only the
fixed contraction $D_N$; no later estimate will depend on $N$.

Since ${\bf 1}+y_0$ lies in a fixed neighbourhood of the identity multiplier, it
is invertible as a~multiplier and both it and its inverse are bounded by an
absolute constant.  Andrew's proof of \cite[Proposition~2.5]{Andrew1981},
applied to this fixed neighbourhood, gives an absolute constant $K$ with the
following property: for every sufficiently far-out increasing sequence $(m_k)$,
the sequence
\[
  u_k=P_{m_k}({\bf 1}+y_0)\qquad(k\in\N)
\]
is $K$-equivalent to $(s_k)$ and admits a left inverse $L$ with
\[
  Lu_k=s_k\qquad(k\in\N),
  \qquad \norm L\le K.
\]
The important point is that $K$ depends only on the fixed neighbourhood of the
identity multiplier, not on the original vector $y$ or on the discarded integer
$N$.

We now run Andrew's gliding-hump argument from
\cite[Proposition~2.7]{Andrew1981} on the sequence $(\widetilde z_n)$.  After
passing to a subsequence, we may choose increasing integers $m_k$ and successive
block vectors $w_k$ so that
\[
  \sum_{k=1}^\infty
  \norm{\widetilde z_{n_k}-(u_k+w_k)}<\varepsilon,
\]
where $\varepsilon>0$ will be chosen below.  Andrew's norm-one projection from
\cite[Lemma~2.6]{Andrew1981} annihilates the block part and fixes the $u_k$'s;
denote it by $Q$.  Thus $Q(u_k+w_k)=u_k$ for every $k$.  Put $R_0=LQ$.  Then
$R_0(u_k+w_k)=s_k$ and $\norm{R_0}\le K$.  Taking $\varepsilon$ so small that
\[
  \sum_{k=1}^\infty\norm{R_0\widetilde z_{n_k}-s_k}<1/2,
\]
the small-perturbation principle for the basic sequence $(s_k)$ gives an
automorphism $A$ of $J$ with $A R_0\widetilde z_{n_k}=s_k$ for every $k$ and
$\norm A\le2$.  Hence $R=AR_0D_N$ satisfies $Rz_{n_k}=s_k$ and has norm at most
$2K$ in the normalised case $\lambda=1$, and at most $2K\delta^{-1}$ in general.

Andrew's proof is written for real scalars, but the estimates above use only absolute values, coordinate projections and bounded multipliers.  The same proof therefore applies to the complex James space.
\end{proof}

\begin{theorem}[The James space]\label{thm:James-UPFP}
The complex James space $J$ has the UPFP.  Moreover,
\[
  \MJ=\Wideal(J)
  \qquad\hbox{and}\qquad
  \Bop(J)/\MJ\simeq\C,
\]
where $\Wideal(J)$ denotes the ideal of weakly compact operators on $J$.
\end{theorem}

\begin{proof}
Laustsen proved that $\Wideal(J)$ is the unique maximal ideal of $\Bop(J)$; since $J$ is quasi-reflexive of order one, this ideal has codimension one \cite[Theorem~4.16]{Laustsen2002}.  Thus every $T\in\Bop(J)$ can be written uniquely in the form
\[
  T=\lambda \Id_J+W
  \quad \big(\lambda\in\C,
  W\in\Wideal(J)\big).
\]
Let $T\in\Bop(J)$.  If $|\lambda|\ge1/2$, apply Lemma~\ref{lem:Andrew-uniform} to the bounded sequence $z_n=Ts_n$.  Indeed, $s_n\to {\bf 1}$ weak-star in $J^{**}$, and weak compactness of $W$ gives
\[
  Ts_n\longrightarrow \lambda {\bf 1}+y
  \quad\hbox{weak-* in }J^{**}
\]
for some $y\in J$.  Hence there are $n_1<n_2<\ldots$ and $R\in\Bop(J)$, with $\norm R\le C_{1/2}$, such that
\[
  R T s_{n_k}=s_k\qquad(k\in\N).
\]
Let $V\in\Bop(J)$ be defined by $Vs_k=s_{n_k}$.  Here we use the spreading, equivalently subsymmetric, nature of the summing basis $(s_n)$; with the usual normalisation we have $\norm V=1$.  Therefore
\[
  R T V=\Id_J.
\]
If $|\lambda|<1/2$, then $|1-\lambda|\ge1/2$, and the preceding argument applied to $\Id_J-T$ gives a~factorisation of $\Id_J$ through $\Id_J-T$ with the same bound.  Hence $J$ has the UPFP.

It remains to identify the ideal $\MJ$.  Since $J$ has the UPFP, it has the PFP, so Proposition~\ref{prop:PFP-maximal-ideal} says that $\MJ$ is the largest proper ideal of $\Bop(J)$.  No weakly compact operator can factor the identity on the non-reflexive space $J$, and hence $\Wideal(J)\subseteq\MJ$.  Conversely, Laustsen's uniqueness theorem says that every proper closed ideal is contained in $\Wideal(J)$.  Thus $\MJ=\Wideal(J)$, and the quotient is $\C$.
\end{proof}

\begin{corollary}\label{cor:James-scalar-answer}
There is an infinite-dimensional complex Banach space $E$ with the UPFP such that
\[
  \bigl(\Bop(E)/\ME\bigr)^{\mathcal U}
\]
is not purely infinite for every ultrafilter $\mathcal U$.
\end{corollary}

\begin{proof}
Take $E=J$.  By Theorem~\ref{thm:James-UPFP}, $\Bop(J)/\MJ\simeq\C$.  The ultrapower of $\C$ is again $\C$, and $\C$ is not purely infinite by Definition~\ref{def:pure-infinite}.
\end{proof}

\begin{remark}
The James-space example is scalar. It shows that the scalar alternative persists even for infinite-dimensional spaces with the UPFP, but it does not affect the non-scalar problem in Question~\ref{q:non-scalar-UPFP}.
\end{remark}


\begin{thebibliography}{99}

\bibitem{Acuaviva2026LpCK}
A.~Acuaviva,
\emph{Primariness of the spaces $\ell_p(C(K))$ for $1\le p\le\infty$},
arXiv:2605.29854 (2026).
\url{https://doi.org/10.48550/arXiv.2605.29854}.

\bibitem{AcuavivaKania2026}
A.~Acuaviva and T.~Kania,
\emph{Primariness and the primary factorisation property},
arXiv:2605.21711 (2026).
\url{https://doi.org/10.48550/arXiv.2605.21711}.

\bibitem{Andrew1981}
A.~D. Andrew,
\emph{Spreading basic sequences and subspaces of James' quasi-reflexive space},
Math. Scand. \textbf{48} (1981), 109--118.
\url{https://doi.org/10.7146/math.scand.a-11904}.

\bibitem{AndrewGreen1980}
A.~D. Andrew and W.~L. Green,
\emph{On James' quasi-reflexive Banach space as a Banach algebra},
Canad. J. Math. \textbf{32} (1980), no.~5, 1080--1101.
\url{https://doi.org/10.4153/CJM-1980-083-7}.

\bibitem{bonsall}
F.~F.~Bonsall and J.~Duncan,
\emph{Complete Normed Algebras},
Springer-Verlag, New York, 1973.

\bibitem{caradus}
S.~R.~Caradus, W.~E.~Pfaffenberger and B.~Yood,
\emph{Calkin Algebras and Algebras of Operators on Banach Spaces},
Marcel Dekker, Inc., New York, 1974.

\bibitem{Daws2006}
M.~Daws,
\emph{Closed ideals in the Banach algebra of operators on classical non-separable spaces},
Math. Proc. Cambridge Philos. Soc. \textbf{140} (2006), no.~2, 317--332.
\url{https://doi.org/10.1017/S0305004105009102}.

\bibitem{DawsHorvath2022}
M.~Daws and B.~Horv\'ath,
\emph{A purely infinite Cuntz-like Banach $*$-algebra with no purely infinite ultrapowers},
J. Funct. Anal. \textbf{283} (2022), no.~1, Paper No.~109488, 36 pp.
\url{https://doi.org/10.1016/j.jfa.2022.109488}.

\bibitem{GowersMaurey}
W.~T.~Gowers and B.~Maurey,
\emph{Banach spaces with small spaces of operators},
Math. Ann. \textbf{307} (1997), 543--568.
\url{https://doi.org/10.1007/s002080050050}.

\bibitem{Horvath}
B.~Horv\'ath,
\emph{Banach spaces whose algebras of operators are Dedekind-finite but they do not have
stable rank one}, Banach Algebras and Applications (Proceedings of the International Conference held at the University of Oulu, July 3-11, 2017) (editor: M. Filali). De Gruyter (2020), 165--176.
\url{https://doi.org/10.1515/9783110602418-009}.

\bibitem{KaniaKoszmiderLaustsen2015}
T.~Kania, P.~Koszmider and N.~J. Laustsen,
\emph{Banach spaces whose algebra of bounded operators has the integers as their \(K_0\)-group},
J. Math. Anal. Appl. \textbf{428} (2015), no.~1, 282--294.
\url{https://doi.org/10.1016/j.jmaa.2015.03.021}.

\bibitem{KaniaLaustsen2012}
T.~Kania and N.~J. Laustsen,
\emph{Uniqueness of the maximal ideal of the Banach algebra of bounded operators on $C([0,\omega_1])$},
J. Funct. Anal. \textbf{262} (2012), no.~11, 4831--4850.
\url{https://doi.org/10.1016/j.jfa.2012.03.011}.

\bibitem{Koszmider2004}
P.~Koszmider,
\emph{Banach spaces of continuous functions with few operators},
Math. Ann. \textbf{330} (2004), no.~1, 151--183.
\url{https://doi.org/10.1007/s00208-004-0544-z}.

\bibitem{Laustsen2002}
N.~J. Laustsen,
\emph{Maximal ideals in the algebra of operators on certain Banach spaces},
Proc. Edinburgh Math. Soc. (2) \textbf{45} (2002), no.~2, 523--546.
\url{https://doi.org/10.1017/S0013091500001097}.

\bibitem{Laustsen2003}
N.~J. Laustsen,
\emph{On ring-theoretic (in)finiteness of Banach algebras of operators on Banach spaces},
Glasgow Math. J. (1) \textbf{45} (2003), no.~1, 11--19.
\url{https://doi.org/10.1017/S0017089502008947}.

\bibitem{LaustsenSkillicorn2017}
N.~J. Laustsen and R.~Skillicorn,
\emph{Extensions and the weak Calkin algebra of Read's Banach space admitting discontinuous derivations},
Studia Math. \textbf{236} (2017), no.~1, 51--62.
\url{https://doi.org/10.4064/sm8554-9-2016}.

\bibitem{LoyWillis1989}
R.~J. Loy and G.~A. Willis,
\emph{Continuity of derivations on $\Bop(E)$ for certain Banach spaces $E$},
J. London Math. Soc. (2) \textbf{40} (1989), no.~2, 327--346.
\url{https://doi.org/10.1112/jlms/s2-40.2.327}.

\bibitem{Read1989}
C.~J. Read,
\emph{Discontinuous derivations on the algebra of bounded operators on a Banach space},
J. London Math. Soc. (2) \textbf{40} (1989), no.~2, 305--326.
\url{https://doi.org/10.1112/jlms/s2-40.2.305}.

\bibitem{Tarbard}
M. Tarbard,
\emph{Operators on Banach Spaces of Bourgain-Delbaen Type},
Ph.D. Thesis, University of Oxford (2013).
\url{https://arxiv.org/abs/1309.7469}.

\end{thebibliography}
\end{document}